\documentclass{amsart}
\usepackage{amsmath}
\usepackage[T1]{fontenc}
\usepackage{textcomp} 
\DeclareEncodingSubset{TS1}{hlh}{1}  

\usepackage{amscd,amssymb,epsfig, epsf,pinlabel,graphicx}
\usepackage{epic,eepic}
\usepackage[dvipsnames]{xcolor}

\usepackage[all]{xy}
\SelectTips{cm}{}
\allowdisplaybreaks

\numberwithin{equation}{subsection}

\newtheorem{propo}{Proposition}[section]

\newtheorem{theor}[propo]{Theorem}
\newtheorem{lemma}[propo]{Lemma}

\theoremstyle{definition}

\theoremstyle{remark}

\let\oldmarginpar\marginpar
\renewcommand\marginpar[1]{\oldmarginpar{\footnotesize #1}}

\newcommand{\ZZ}{\mathbb{Z}}

\begin{document}

\title[Multi-racks]{Multi-quandles of topological pairs}

    \author[Vladimir Turaev]{Vladimir Turaev}
    \address{
    Vladimir Turaev \newline
    \indent   Department of Mathematics \newline
    \indent  Indiana University \newline
    \indent Bloomington IN47405, USA\newline
    \indent and  \newline
    \indent IRMA, Strasbourg \newline
    \indent 7 rue Rene Descartes \newline
    \indent 67084 Strasbourg, France \newline
    \indent $\mathtt{vturaev@yahoo.com}$}

\begin{abstract}  In generalization of knot quandles we introduce  similar algebraic structures associated with  arbitrary pairs  consisting of a  path-connected topological space and its path-connected subspace.
\end{abstract}

\maketitle

\section {Introduction}

A quandle is a set with a binary operation whose axioms are inspired by the Reidemeister moves on knot diagrams. Every knot in the 3-sphere gives rise to a quandle which can be defined  in terms of    diagrams of the knot or in terms of homotopy classes of paths in the knot exterior. The study of  quandles has been an active  part of knot theory  since the fundamental papers of Joyce \cite{Jo} and Matveev \cite{Ma}, see  a few sample papers on this subject in the list of references.

This paper was inspired by the following  question: can one  define analogs of knot quandles for arbitrary topological pairs? More precisely, consider a pair $(X,Y \subset X)$ of path-connected topological  spaces. Pick a base point $x\in X$. By a \emph{$Y$-path} we shall mean a continuous path in~$X$ starting in~$x$ and terminating in~$Y$. By \emph{$Y$-homotopy} (or just \emph{homotopy}) of $Y$-paths we mean a continuous deformation in the class of $Y$-paths.  During such a deformation, the starting point of the path does not move while the terminal point  may slide along~$Y$ in an arbitrary way. Consider the set $\pi(X,Y,x)$ of the homotopy classes of   $Y$-paths.  We ask
whether this set carries a natural quandle-type  structure.

The answer turns out to be positive with a few caveats.  First, instead of quandles we involve more general algebraic structures which we call multi-quandles. Note that the multi-quandles of  the pair $(X,Y)$ corresponding to different choices of the base point $x\in X$ are  isomorphic.  Second, the  isomorphism class of these multi-quandles  is entirely determined by the fundamental groups of $X,Y$ and the inclusion homomorphism $\pi_1(Y)\to \pi_1(X)$.  This, one can say, is a disappointment as we get no really new invariants of the pair $X, Y$. Third, in this general setting we do not have knot diagrams which makes it impossible to apply diagrammatic techniques  used  in the theory of knot quandles. Nevertheless, the study of multi-quandles suggests a number of interesting  questions. Presumably,  various algebraic methods developed for quandles can be extended to multi-quandles.

The paper starts with the  definition   of so-called  multi-racks. Then we define multi-quandles, discuss    a few  constructions of multi-quandles, and define multi-quandles of topological pairs. 

Recently, algebraic objects analogous to multi-racks and multi-quandles were  independently introduced by  V. G. Bardakov and D. A. Fedoseev, see \cite{BF}.


\section{Multi-racks and multi-quandles}\label{brbrr}

We  define multi-racks, multi-quandles and give a few simple examples.
 
 \subsection{Basics}\label{topolsssett}\label{topddolsXett} By  a \emph{binary operation} in a set~$U$  we  mean an arbitrary mapping $\rhd: U\times U \to U$. The image of a pair $(u,v)\in U\times U$  under this mapping is denoted $u\rhd v$.  A  binary operation~$\rhd$  in~$U$ is \emph{non-degenerate} if for any $v\in U$, the map $U\to U, u\mapsto u\rhd v$ is bijective.  A \emph{multi-rack} is a pair consisting of a non-empty set~$U$ and a collection $\{\rhd_s\}_{s\in  S}$ of  non-degenerate  binary operations in~$ U$ such that for any $s ,t\in S, u,v,w\in U$ we have    
 \begin{equation}\label{Jaco}(u \rhd_s v)\rhd_t w=(u \rhd_t w)\rhd_s (v\rhd_t w) \end{equation} and
 $$(u \rhd_t v)\rhd_s w=(u \rhd_s w)\rhd_t (v\rhd_sw) .$$
 In particular,  for all $s\in S$ and   $u,v,w \in U$ we must  have
 \begin{equation}\label{mleftJaco} (u \rhd_s v)\rhd_s w=(u\rhd_s w)\rhd_s (v\rhd_s w).\end{equation}
We call~$U$ the \emph{underlying set} of the multi-rack $(U, \{\rhd_s\}_{s\in  S})$  and call~$S$ the set of binary operations in this multi-rack. Of course, if~$T$ is a subset of~$S$, then $(U, \{\rhd_s\}_{s\in  T})$ is also a multi-rack. 
 
 A \emph{morphism} of  multi-racks  $(U, \{\rhd_s\}_{s\in  S}) \to (V, \{\rhd_t\}_{t\in  T})$ is a pair of maps $\Phi:U\to V, \varphi:S \to T$ such that for any $u,v\in U, s\in S$ we have
$$\Phi(u) \rhd_{\varphi(s)} \Phi (v)=\Phi (u\rhd_s v).$$
 Multi-racks and their morphisms form a category with respect to the obvious composition of  morphisms and  the obvious identity morphisms.

A multi-rack $(U, \{\rhd_s\}_{s\in  S})$ is a \emph{multi-quandle} if $u\rhd_s u =u$ for all $u\in U, s\in S$. A multi-rack (respectively, a multi-quandle) $(U, \{\rhd_s\}_{s\in  S})$  with $\rm{card} (U)=1$ is a \emph{rack} (respectively, a \emph{quandle}).

 \subsection{Examples}\label{exa} 1.  Any non-empty set~$U$ with the binary operation $u\rhd v=u$ for all $u,v\in U$ is a quandle.
 
 2.  Consider a group~$G$ and a  family of commuting automorphisms $\{f_s: G\to G\}_{s\in S}$ of~$G$. For  $s\in S$, define a map $\rhd_s:G\times G \to G$ by   $$u \rhd_s v=f_s(uv^{-1})v$$ for all $u,v\in G$. Since $f_s$ is an automorphism of~$G$, the binary operation $\rhd_s$ is non-degenerate. The identity $u\rhd_s u=u$ is obvious.
 Direct computations show that for any $s, t\in S, u,v,w\in G$, 
 $$(u\rhd_s v)\rhd_{t}w=f_s(uv^{-1})v \rhd_{t}w=f_t (f_s(uv^{-1})vw^{-1} )w$$
 and $$(u\rhd_{t} w)\rhd_s (v\rhd_{t} w)=f_t(uw^{-1})w\rhd_s f_t(vw^{-1})w$$
 $$=f_s(f_t(uw^{-1})w( f_t(vw^{-1})w)^{-1})  f_t(vw^{-1})w $$
 $$=f_s(f_t(uv^{-1}))f_t(vw^{-1} )w=f_t (f_s(uv^{-1})vw^{-1} )w$$
 where we use that $f_{s},f_{t}$ are commuting automorphisms of~$G$. Hence, the pair $(G, \{\rhd_s\}_{s\in S})$ is a multi-quandle.  To give a  specific example, we can take as~$G$ a module over a commutative ring and as $\{f_s\}_{s\in S}$  multiplications by  invertible elements of this ring. 
 
 3. Consider a group~$G$  and for each integer~$n$ define a  binary operations $\rhd_n$ in~$G$ as follows.   For any $u,v \in G$,    set
 \begin{equation}\label{ope} u \rhd_{n} v= v^{-n} u v^n\in G. \end{equation}
 Direct computations show that the pair $(G, \{\rhd_n\}_{n\in \ZZ})$ is a  multi-quandle.
 
 


  \section{Constructions of multi-racks and multi-quandles}\label{gtmrrbrr}
 
  The following  constructions of multi-racks and multi-quandles play the key role in the definition of multi-quandles of topological pairs.

 \subsection{Multi-racks}\label{gmrexa}  
  For any  group~$G$ we define a multi-rack having~$G$ as  both the underlying    set and   the set of binary operations.  Namely, for any $u, v, s\in G$,   set
 \begin{equation}\label{ope}u \rhd_s v= vsv^{-1} u \in G. \end{equation}

  
  \begin{lemma}\label{sDDfffDtrbucte}    The pair $(G, \{\rhd_s\}_{s\in G})$ is a  multi-rack.
   \end{lemma}
  
\begin{proof} The non-degeneracy of $\rhd_s$ is obvious. For any $u,v, w, s, t\in G $, we have
$$(u \rhd_s v)\rhd_{t} w=vsv^{-1} u\rhd_{t} w= wt w^{-1} vsv^{-1} u$$
and $$(u \rhd_{t} w)\rhd_s (v\rhd_{t} w)=wtw^{-1} u \rhd_s
wtw^{-1} v=$$ 
$$=wtw^{-1} v s (wtw^{-1} v)^{-1} wtw^{-1} u=wt w^{-1} vsv^{-1} u.$$ 
This proves  \eqref{Jaco}.
  \end{proof}

 Note that if  $G\neq\{1\}$ then the multi-rack  in Lemma~\ref{sDDfffDtrbucte} is not a multi-quandle: for any $s\in G\setminus \{1\}$ we have $s\rhd_s s=s^2 \neq s$.

  \subsection{Multi-quandles}\label{cosetgmrexa} Every pair (a group~$G$, a subgroup $H\subset G$) gives rise to a multi-quandle as follows.    
Let $G/H$ be the set of  left cosets of~$H$ in~$G$ and let $p=p_H:G\to G/H$ be the  projection 
  carrying any $u\in G$ to the left coset $uH$. 
  Let  $\{\rhd_s\}_{s\in G}$  be the  binary operations    \eqref{ope}  in~$G$.

    \begin{lemma}\label{sDucte}    For any element~$s$ of the center of~$H$,   there is a unique  binary operation  $\rhd^\circ_s$ in $G/H$ such that the diagram
\begin{equation}\label{dia1}
\xymatrix{
G\times G \ar[rr]^-{\rhd_s} \ar[d]_{{ p}  \times   p} & & G\ar[d]^-{{  p}}\\
{G/H\times G/H} \ar[rr]^-{\rhd^\circ_s}&  & G/H
} \end{equation}
commutes. 
    \end{lemma}
    
    \begin{proof} Pick any $u,v\in G$ and  $g,h \in H$.
    By definition,
   $$ ug\rhd_s vh = vh s (vh)^{-1}ug=vh sh^{-1} v^{-1}ug=vsv^{-1} ug=(u\rhd_s v)g$$ where we use 
   that $hs=sh$. Consequently,  $$  p(ug\rhd_s vh)=p(u\rhd_s v).$$   
   Now it is clear that the formula  $$uH\rhd^\circ_s vH= (u\rhd_s v)H$$ defines the  unique binary operation $\rhd^\circ_s$ in~$G/H$ satisfying the conditions of the lemma. \end{proof}

   We shall denote the center of a group~$H$ by $Z(H)$.

      \begin{theor}\label{sDuctet}    Under the assumptions of Lemma~\ref{sDucte}  the   pair
      $$(G/H, \{\rhd^\circ_s\}_{s\in Z(H)}\}$$ is a  multi-quandle. \end{theor}
   
   \begin{proof} Pick   $s,t\in Z(H)$ and $u,v,w\in G$.  Lemma~\ref{sDucte} implies that
   $$p((u\rhd_sv)\rhd_t w)=p(u\rhd_s v)\rhd^\circ_t p(w)=(p(u)\rhd^\circ_s p(v))\rhd^\circ_t p(w).
$$ 
Similarly, $$
p((u \rhd_t w)\rhd_s (v\rhd_t w))= (p(u) \rhd_t^\circ p(w))\rhd_s^\circ (p(v)\rhd_t^\circ  p(w)) . $$   
Now, the equality \eqref{Jaco} in~$G$ implies that $$(p(u)\rhd^\circ_s p(v))\rhd^\circ_t p(w)=
(p(u) \rhd_t^\circ p(w))\rhd_s^\circ (p(v)\rhd_t^\circ  p(w)) . $$ Since the projection $p:G\to G/H$ is surjective, we conclude that the equality \eqref{Jaco} holds for all elements of $G/H$ and all $s,t\in Z(H)$.

Next we observe that 
$$p(u) \rhd_s^\circ p(u)=p(u\rhd_s u)=p (usu^{-1}u)=p(us)=p(u)$$
where the last equality holds because $s\in H$. Thus, $p(u) \rhd_s^\circ p(u)=p(u)$. 
The surjectivity of~$p$ implies that $x\rhd^\circ_s x=x$ for all $x\in G/H$. 

It remains to prove the non-degeneracy of the operations $\{\rhd^\circ_s\}_s$ in $G/H$. We must show that for any $s\in Z(H), v\in G$, the  map $f:G/H\to G/H$ defined  by $f(x)= x \rhd^\circ_sp(v)$ is bijective.   First, we prove the surjectivity. 
    Denote the bijection $u\mapsto u \rhd_s v:G\to G$ by~$F$.
     For any $w\in G$, we have $$f (p(F^{-1} (w)) )= p(F^{-1} (w) ) \rhd^\circ_s  p(v)=p( F^{-1} (w) \rhd_s  v)=p( FF^{-1} (w) )=p(w).$$ Therefore $p(w)\in {\rm {Im}} f$ and~$f$ is surjective. Suppose now that two elements $x_1,x_2$ of $G/H$ satisfy $f(x_1)=f(x_2)$. 
     For $i=1,2$, pick any $w_i\in p^{-1}(x_i)\subset G$.
     Then $$p(w_i \rhd_s v)= p(w_i)\rhd^\circ_s p(v)=x_i \rhd^\circ_s p(v)= f(x_i).$$ 
     The equality $f(x_1)=f(x_2)$ implies that the elements $\{w_i \rhd_s v\}_{i=1,2}$ of~$G$ lie in the same left coset of~$H$.  By the definition of $\rhd_s$, this means that the elements $\{vsv^{-1}w_i \}_{i=1,2}$ of~$G$ lie in the same left coset of~$H$. Consequently, $w_1, w_2\in G$ lie in the same left coset of~$H$. Therefore $x_1=p(w_1)=p(w_2)=x_2$. We conclude that the mapping~$f$ is injective. This fact and the surjectivity of~$f$ implies that~$f$ is a bijection. This completes the proof of the non-degeneracy of $\rhd^\circ_s$ and  of the theorem. 
     \end{proof}

  \section{Multi-quandles in topology}\label{gttoppmrrbrr}

Consider a pair $(X,Y \subset X)$ of non-empty path-connected topological  spaces. For any  base point $x\in X$, we define a structure of a multi-quandle in the set 
 $\pi(X,Y,x)$ of $Y$-homotopy classes of $Y$-paths leading from~$x$ to~$Y$, as defined in the introduction. We also explain that our construction includes the standard knot quandles as a special case.
 
  \subsection{The case $x\in Y$}\label{cosetgmgennrexa} The case where $x\in Y$ is  simpler and we start with it.
  Set $G=\pi_1(X,x)$ and let   $H\subset G$ be the image of the inclusion homomorphism
$ \pi_1(Y,x) \to \pi_1(X, x)=G$. 
   We  first construct a bijection $G/H\to \pi(X,Y,x)$. 
 Since $x\in Y$, any path in~$X$ starting and ending in~$x$ is a $Y$-path. Considering the homotopy  classes of such paths  we obtain a  mapping $P:G\to \pi(X,Y,x)$. The path-connectedness of~$Y$ implies that the mapping~$P$ is surjective. Our definitions ensure that~$P$
 is a composition of the projection $p:G\to G/H$ and a mapping $G/H\to \pi(X, Y, x)$. The latter mapping is  bijective: its inverse  is obtained by assigning to any $Y$-path $\gamma:[0,1]\to X$  the left coset $p([\gamma \beta])\in G/H$  where $\beta$ is any path in~$Y$  from  $\gamma(1)\in Y$ to~$x$ and $[\gamma \beta]\in G$ is the homotopy class of the loop $\gamma \beta$.
 
Theorem~\ref{sDuctet} yields a structure of a multi-quandle in  $G/H$.   Transporting this  structure along the bijection $G/H\to \pi(X, Y, x)$ we obtain a structure of a multi-quandle in the set $\pi(X, Y, x)$ with binary operations numerated by elements of the center of the group~$H$.

   \subsection{The general case}\label{cosetgmgennrexa} 
Suppose now that~$x$ is an arbitrary point of~$X$. Fix a path~$\alpha$ in~$X$ leading from~$x$ to a point $y\in Y$. This path determines a bijection $\pi(X, Y,y) \to \pi(X,Y,x)$ carrying the $Y$-homotopy class of any path~$\rho$ from~$y$ to~$Y$ into the $Y$-homotopy class of the product path $\alpha\rho$ from~$x$ to~$Y$. Transporting the multi-quandle  structure in $\pi(X,Y,y)$  constructed above to $\pi(X, Y, x)$ along this bijection we obtain a multi-quandle structure in  $\pi(X, Y, x)$. 
Generally speaking, the latter structure  depends on the choice of the path~$\alpha$. However, the multi-quandles associated with different $x,y,\alpha$ are  isomorphic (exercise). We  denote the  isomorphism class of these multi-quandles  by $Q(X,Y)$.

  \subsection{The knot quandles}\label{cokcnnrexa} Consider a path-connected manifold~$M$ of dimension $\geq 2$ and its  codimension 2 path-connected submanifold $K \subset {\rm{Int}}(M)$. We assume that  the normal bundle of~$K$ in~$M$ is   trivial and oriented. Let $E=E_K$ be the exterior of~$K$ in~$M$. The manifold $\partial E$ consists of $\partial M$ and the product $K\times S^1$ bounding  a regular neighborhood of~$K$ in~$M$. Applying  the constructions above to the pair $E\supset  K\times S^1$, we obtain a multi-quandle $Q=Q(E,K\times S^1)$. The  binary operations in~$Q$ are numerated by the elements of the  abelian group $Z(H)$ where~$H$ is the image of the inclusion homomorphism $\pi_1(K\times S^1) \to \pi_1(E)$.  The group $Z(H)$ has a distinguished element represented by the meridian circle of~$K$. Keeping only the corresponding binary operation in $Q$ (and forgetting the binary operations corresponding to all other elements of $Z(H)$), we obtain  the Joyce-Matveev quandle of the pair $(M,K)$. This construction applies in particular to knots in $S^3$ and yields the Joyce-Matveev quandles of knots.

    \end{document}